\documentclass[a4paper, preprint,10pt]{imsart}
\usepackage[english]{babel}
\usepackage[utf8x]{inputenc}
\usepackage[margin=1.2in]{geometry}
\usepackage{amsthm,amsmath, amssymb}
\usepackage{graphicx}
\usepackage[colorinlistoftodos]{todonotes}

\startlocaldefs
\numberwithin{equation}{section}
\theoremstyle{plain}
\newtheorem{thm}{Theorem}[section]
\theoremstyle{definition}
\newtheorem{rem}{Remark}[section]
\theoremstyle{definition}
\newtheorem{dfn}{Definition}[section]
\theoremstyle{lemma}
\newtheorem{lem}{Lemma}[section]
\theoremstyle{proposition}
\newtheorem{pro}{Proposition}[section]
\theoremstyle{corollary}
\newtheorem{cor}{Corollary}[section]
\endlocaldefs

\begin{document}
\text{ Previously titled: "On the sum of independent random variables of the exponential family"} \\

\title{On the density for sums of independent exponential, Erlang and gamma variates}
\runtitle{On the density for sums of independent exponential, Erlang and gamma variates}
\author{\textsc{Edmond Levy} }
\begin{aug}
\\
\textnormal{Stanmore, Middlesex, UK}
\\
\textnormal{elevy123@outlook.com}
\end{aug}

\begin{abstract}
This paper re-examines the density for sums of independent exponential, Erlang and gamma random variables. By using a divided difference perspective, the paper provides a unified approach to finding closed-form formulae for such convolutions. In particular, the divided difference perspective for sums of Erlang variates suggests a new approach to finding the density for sums of independent gamma variates using fractional calculus.
\end{abstract}


\maketitle
\textnormal{ }
Version: 15 July 2021
\\
\\
\section{Introduction}
Exponentially distributed random variables are prevalent in the applied fields of probability and in stochastic modelling. In perhaps the most celebrated of such models, the Poisson process, the interarrival times of events are independent exponential random variables as a consequence of its postulates. Combining such processes in the development of his \textit{method of stages} (or phases) in queueing models, Agner Erlang was led to introduce what is now the familiar \textit{Erlang distribution} being the distribution of an integer sum of independent and identically distributed exponentials. The gamma density is a less restricted form of the Erlang density where the shape parameter may take an arbitrary positive noninteger value. 

Such random variables and their sums lie at the core of many fields such telecommunications, statistics, reliability theory and survival and risk analysis, to name a few. We mention here a few specific examples. In the field of reliability, the uncertain lifetimes of components are often modelled as exponential variates (see \cite{Amari1997}) and hence system failure times are distributed as a sum of exponentially distributed random variables. Referring to the sum of heterogeneous Erlang variables as a \textit{Generalised Integer} gamma variable, \cite{Coelho1998} shows its relevance to Wilks' lambda when testing the independence of sets of normally distributed variables. Finally, in \cite{Sim1992}, Sim develops point processes for the modelling of non-Poisson series of events where the time to the $k$th event is interpreted as the sum of $k$ independent but non-identically distributed gamma variables.

The paper begins with the formula for the \textit{hypo-exponential} density being that for the sum of independent exponentials having pairwise distinct parameters. We point out that this density has a divided difference characteristic which immediately suggests a novel perspective from which to further explore the densities of sums of independent exponentials. The paper advances a succinct representation for the density of independent Erlang distributed variables and demonstrates agreement with previous papers where such formulae have been found (and often rediscovered) by other means. Extending these results further, by using the tools of fractional calculus, a representation is also found for the density for sums of distinct independent gamma random variables. The paper concludes by showing how this approach produces the density function itself.

The divided differences is a subject more familiar in the fields of approximation theory and numerical analysis. As a fundamental tool, it has been put to good effect in simplifying and elucidating formulae for the moments of the geometric Brownian motion, see \cite{Baxter2011}. It is hoped that the ideas here demonstrate another aspect of its usefulness and provide the reader with new avenues to explore the densities for sums of exponential and gamma random variables further.

\section{The hypo-exponential density}
Let $X_i$ be a random variable having the exponential distribution with rate (or intensity) paramenter $\lambda_i >0$. Then its probability density function, $f_{X_i}(t)$, is given by:
\begin{equation}\notag
f_{X_i}(t) = 
\begin{cases}
\lambda_i e^{-\lambda_it} &\quad t \ge 0 \\
0 &\quad t < 0.
\end{cases}
\end{equation}
The sum of $n$ mutually independent exponential random variables, $X_i$, with pairwise distinct parameters, $\lambda_i$, $i=1, \dots, n$, respectively, has the hypo-exponential density, $S_n(t)$, given by
\begin{equation}
S_n(t)=\Big(\prod_{i=1}^{n} \lambda_i \Big)\sum_{j=1}^n\frac{e^{-\lambda_jt}}{\prod\limits_{\substack{k=1 \\ k\neq j}}^n (\lambda_k-\lambda_j)}, \quad t\ge 0.
\end{equation}
We note that the condition that the $\lambda_i$'s be distinct is essential as the formula (2.1) is undefined for any instance where $\lambda_i=\lambda_j$ for $i \ne j$. This formula is well known and its derivation can be found in a number of sources, for example \cite{Ross1997}. 

To an approximation theorist, the form of (2.1) is very familiar. It is the $(n-1)$th-order divided difference of the function $e(x)=e^{xt}$ at the points $-\lambda_1, \dots, -\lambda_n$, expressed in its Lagrange polynomial form, multiplied by the product of all the $\lambda_i$'s. Knowing this, suggests an alternative perspective and a common basis when extending to more general instances where the $\lambda_i$'s are not distinct and some, or all, parameters have repeats.

\section{Preliminaries}

\subsection{Newton's divided differences}

Given data points $(x_i, y_i)$ for $i=1, \dots, m$, a standard interpolation problem is to approximate the (possibly unknown) function $y=f(x)$ generating the data points by a known function constructed to pass though each such data point. Newton's method proceeds by casting the problem as determining the coefficients $b_0, \dots,b_{m-1}$ under a recursive scheme of polynomials of increasing order:
\begin{equation}
q_i(x) = q_{i-1}(x)+b_{i-1}(x-x_1)(x-x_2) \dots (x-x_{i-1}), \quad i=2, \dots, m,
\end{equation}
beginning with $q_1(x)=b_0$. If we regard these as a system of $m$ equations, we may find the coefficients by solving for the column vector $\textbf{b}=(b_0, \dots, b_{m-1})'$ in the matrix equation system:
\begin{equation}
\textbf{y}=\textbf{Tb}
\end{equation}
where, $\textbf{y}=(y_1, \dots, y_m)'$ and \textbf{T} is the matrix
\begin{equation}\notag
\begin{bmatrix} 1 & 0 & 0 & \ldots & 0 \\ 1 & (x_2-x_1) & 0 & \ldots & 0 \\ 1 & (x_3-x_1) & (x_3-x_2)(x_3-x_1) & & 0 \\ \vdots & \vdots & \vdots & \ldots & \vdots \\ 1 & (x_m-x_1) & (x_m-x_2)(x_m-x_1) & \ldots & \prod_{k=1}^{m-1} (x_m-x_k) \end{bmatrix}.
\end{equation}
When the $x_i$'s are distinct, $\textbf{T}$ is nonsingular and the unique solution is $\textbf{b}=\textbf{T}^{-1}\textbf{y}$. The triangular nature of $\textbf{T}$ reflects the recursive form of (3.1). 
It can be shown (see \cite{Atkinson1989} p.140 or \cite{Schatzman2002} Lemma 4.2.2) that the coefficient $b_{k-1}$ $(k=1, \dots, m)$ is (and defines) the \textit{$(k-1)$th-order divided difference} of the function $f(.)$ at points $x_1, \dots, x_k$, i.e. 
\begin{equation}\notag
b_{k-1} \equiv f[x_1,x_2, \dots, x_k],
\end{equation}
and $f[x_1,x_2, \dots, x_k]$ may be given the alternative following definition:
\begin{dfn}
For a function \(f(.)\) defined at points \(x_1,\ldots,x_k\), the $(k-1)$th-order divided difference is defined by the recurrence relation: 
\begin{equation}
f[x_1,\ldots,x_k]=\frac{f[x_2,\ldots,x_{k-1},x_k]-f[x_1\ldots,x_{k-2},x_{k-1}]}{x_k - x_1}
\end{equation}
with \(f[x]=f(x)\). 
\end{dfn}
\noindent
It can also be shown that when the arguments $x_1, \dots, x_k$ are distinct, the divided difference $f[x_1,\ldots,x_k]$ can be expressed in terms of \textit{Lagrange polynomials}:
\begin{equation}
f[x_1,\ldots,x_k]=\sum_{j=1}^k\frac{f(x_j)}{\prod\limits_{\substack{q=1 \\ q\neq j}}^k (x_j-x_q)},
\end{equation}
(see \cite{Atkinson1989} p.139). We note that the form of (3.4) shows \(f[x_1, \ldots, x_m]\) to be a symmetric function of its arguments and so the calculations are invariant to permutations in the order of its arguments, e.g. $f[x_1,x_2,x_3] = f[x_2,x_3,x_1]$. See \cite{Schatzman2002} Lemma 4.2.1. 

\subsection{Lemmas} The following short lemmas will prove helpful.
\begin{lem}
For any $m>1$ distinct points $x_1, \dots, x_m$ we have:
\begin{equation}
\sum_{j=1}^m\frac{1}{\prod\limits_{\substack{k=1 \\ k\neq j}}^m (x_j-x_k)} \equiv 0.
\end{equation}
\end{lem}
\begin{proof}
Examine the system (3.2) for data points $(x_i,y_i)$ but where $y_i = 1$ for $i = 1, \dots, m$. Rather than invert $\textbf{T}$, we instead solve for $\textbf{b}$ using Cramer's rule. Let $\textbf{T}_{i}$ signify the matrix found by replacing the $i$th column of $\textbf{T}$ by $\textbf{y}$ then $b_{i-1} = det(\textbf{T}_{i})/det(\textbf{T})$. Clearly, $\textbf{T}_{i}$ is singular for $i>1$, hence $det(\textbf{T}_{i})=0$ and (thus) $b_{i-1}=0$ for $i=2, \dots, m$. Noting (3.4), we therefore prove (3.5).
\end{proof} 
\noindent
Lemma 3.1 should come as no surprise as it merely states that when divided differences are taken at $m$ points where all function values $f(x_i)$ are equal, $f[x_1, \dots, x_m] \equiv 0$.

\begin{lem}
For distinct values $x_1, \dots, x_m$ and $s \in \mathbb{R}$, the decomposition of the rational function, $U(s)$,
\begin{equation}\notag
U(s)=\frac{1}{\prod\limits_{k=1}^m (x_k - s)},
\end{equation}
as a sum of partial fractions, gives
\begin{equation}
\frac{1}{\prod\limits_{k=1}^m (x_k - s)}=\sum_{j=1}^m \frac{1}{(x_j - s)\prod\limits_{\substack{k=1 \\ k\neq j}}^m(x_k - x_j)}.
\end{equation}
\end{lem}
\begin{proof}
Obvious. (See also \cite{Akkouchi2008}).
\end{proof}

\subsection{Euler's gamma function and related functions} 
We note the following definitions and expressions which are further explored in \cite{Olver2010} and in \cite{Jameson2016}. 
\\
\\
Euler's \textit{gamma function}, $\Gamma(a)$, is defined by:
\begin{equation}\notag
\Gamma(a)=\int_0^{\infty} t^{a-1}e^{-t}dt, \quad \text{ for } Re(a)>0
\end{equation}
and satisfies the relation $\Gamma(a+1)=a\Gamma(a)$. Hence, for any $a \in \mathbb{N}$, $\Gamma(a+1)=a!$. The function may be extended to all $a<0$ except at its poles $\{0, -1, -2, \dots\}$ by defining $\Gamma(a)=\Gamma(a+1)/a$. 
\\
\\
The \textit{incomplete gamma function} $\gamma(a,z)$ for $z \ge 0$ is defined by:
\begin{equation}\notag
\gamma(a,z)=\int_0^z t^{a-1}e^{-t}dt, \quad \text{ for } Re(a)>0.
\end{equation}
The function can also be expressed in series form using the confluent hypergeometric (or Kummer's) function, $M(a,b,z)$, as
\begin{equation}
\begin{split}
\gamma(a,z)&=a^{-1}z^ae^{-z}M(1,1+a, z),  \quad \text{ for } a \ne 0, -1, -2 \dots \\
&=a^{-1}z^aM(a, 1+a, -z),
\end{split}
\end{equation}
using \textit{Kummer's transformation} $M(a,b,z)$=$e^zM(b-a,b,-z)$, (see Eqns.  (8.5.1) and (13.2.39) in \cite{Olver2010}). From the definition of $\gamma(a,z)$, we have $\gamma(1,z)$ = $(1-e^{-z})$, for $Re(z)>0$, and the $n$th derivative of $\gamma(a,z)/z$ is given by:
\begin{equation}
\frac{d^n}{dz^n}\frac{\gamma(a,z)}{z} = (-1)^n\frac{\gamma(n+a,z)}{z^{n+a}},
\end{equation}
see \cite{Jameson2016}. 
\\
\\
For $M(a,b,z)$, we have that:
\begin{equation}
\frac{d^n}{dz^n}M(a,b,z) = \frac{\Gamma(a+n)\Gamma(b)}{\Gamma(a)\Gamma(b+n)}M(a+n,b+n,z),
\end{equation}
(see \cite{Olver2010} Eqn. 13.3.16) and $M(a,b,z)$ has the integral representation
\begin{equation}
M(a,b,z)=\frac{\Gamma(b)}{\Gamma(a)\Gamma(b-a)}\int_0^1e^{zt}t^{a-1}(1-t)^{b-a-1}dt,
\end{equation}
(see \cite{Olver2010} Eqn. 13.4.1). 
\\
\\
The \textit{complementary (or upper) incomplete gamma function} $\Gamma(a,z)$ for $z \ge 0$ is defined by:
\begin{equation}\notag
\Gamma(a,z)=\int_z^{\infty} t^{a-1}e^{-t}dt , \quad \text{ for } Re(a) >0
\end{equation}
and has an alternative integral form (see \cite{Olver2010} Eqn.(8.6.4))
\begin{equation}
\Gamma(a,z)=\frac{z^a}{\Gamma(1-a)}\int_0^{\infty} \frac{t^{-a}e^{-(t+z)}}{t+z}dt , \quad \text{ for } z>0, Re(a) <1.
\end{equation}
Clearly, we have that $\Gamma(a)=\Gamma(a,z)+\gamma(a,z)$, for all $z \ge 0$. When $a$ is an integer $n \ge 1$, $\gamma(n,z)$ and $\Gamma(n,z)$ may be expressed as finite series:
\begin{equation}\notag
\gamma(n,z) =  (n-1)!\big(1-e^{-z}\sum_{r=0}^{n-1} \frac{z^r}{r!}\big) \quad \text{ and } \quad \Gamma(n,z) = (n-1)!e^{-z}\sum_{r=0}^{n-1} \frac{z^r}{r!}.
\end{equation}
Finally, Euler's \textit{beta function}, $B(z,y)$ for $z>0$, $y>0$, is defined by:
\begin{equation}\notag
B(z,y)=\int_0^1 t^{z-1}(1-t)^{y-1}dt 
\end{equation}
and satisfies the relation $B(z,y)=\Gamma(z)\Gamma(y)/\Gamma(z+y)$.

\section{Convolution of exponential random variables with distinct parameters}
The definition and representation of divided differences in Sec. 3.1 immediately suggests the following proposition giving the hypo-exponential density an alternative compact form.
\begin{pro}
Let $X_1, \dots, X_n$ be $n$ independent exponential random variables with parameters $\lambda_i$, $i=1, \dots, n$ and where $\lambda_i \ne \lambda_j$ when $i \ne j$. Let $Y_i = X_1+ \dots + X_i$ and denote its density by $S_i(t)$. Then $Y_n$ has the hypo-exponential density given by
\begin{equation}
S_n(t)=\Big(\prod_{i=1}^{n} \lambda_i \Big)e[-\lambda_1, \dots, -\lambda_n], \quad t\ge 0,
\end{equation}
where $e[-\lambda_1, \dots, -\lambda_n]$ is the $(n-1)$th-order divided difference for the function $e(x)=e^{xt}$ at points $-\lambda_1, \dots, -\lambda_n$.
\end{pro}
\begin{proof}
The proof follows by inspection of (2.1) and recognising the Lagrange polynomial representation (3.4). However, Lemma 3.1 allows for a short proof using induction on $n$. 

For $n=1$ the equation holds trivially. We assume the truth of equation (4.1) at $n-1$ and examine the density for $Y_n = Y_{n-1} + X_n$. Performing the convolution of $f_{X_n}(t)$ with $S_{n-1}(t)$, we have:
\begin{equation}\notag
\begin{split}
S_n(t)&= {\int_0^t\lambda_n e^{-\lambda_n u}S_{n-1}(t-u)\text{d}u} \\
&= {\int_0^t\lambda_n e^{-\lambda_n u} \big(\prod_{i=1}^{n-1} \lambda_i \big)\sum_{j=1}^{n-1}\frac{e^{-\lambda_j(t-u)}}{\prod\limits_{\substack{k=1 \\ k\neq j}}^{n-1} (\lambda_k-\lambda_j)}\text{d}u} \quad (\text{using (3.4) and }\\
&\text{with the understanding that for n=2, } \prod\limits_{\substack{k=1 \\ k\neq 1}}^{n-1} (\lambda_k-\lambda_1) =1)\\
&=\big(\prod_{i=1}^n \lambda_i \big)\sum_{j=1}^{n-1}\frac{e^{-\lambda_jt}{\int_0^te^{-(\lambda_n-\lambda_j)u}\text{d}u}}{\prod\limits_{\substack{k=1 \\ k\neq j}}^{n-1} (\lambda_k-\lambda_j)} \\
&= \big(\prod_{i=1}^n \lambda_i \big)\sum_{j=1}^{n-1}\frac{e^{-\lambda_jt}[1-e^{-(\lambda_n-\lambda_j)t}]}{(\lambda_n-\lambda_j)\prod\limits_{\substack{k=1 \\ k\neq j}}^{n-1} (\lambda_k-\lambda_j)} \\
&= \big(\prod_{i=1}^n\lambda_i \big)\big(\sum_{j=1}^{n-1}\frac{e^{-\lambda_jt}}{\prod\limits_{\substack{k=1 \\ k\neq j}}^n (\lambda_k-\lambda_j)}-\sum_{j=1}^{n-1}\frac{e^{-\lambda_n t}}{\prod\limits_{\substack{k=1 \\ k\neq j}}^n (\lambda_k-\lambda_j)}\big).
\end{split}
\end{equation}
From Lemma 3.1,
\begin{equation}\notag
\sum_{j=1}^n \frac{1}{\prod\limits_{\substack{k=1 \\ k\neq j}}^n (\lambda_k-\lambda_j)} = \sum_{j=1}^{n-1}\frac{1}{\prod\limits_{\substack{k=1 \\ k\neq j}}^n (\lambda_k-\lambda_j)}+\frac{1}{\prod\limits_{k=1}^{n-1} (\lambda_k-\lambda_n)}=0
\end{equation}
hence
\begin{equation}\notag
-\sum_{j=1}^{n-1}\frac{e^{-\lambda_n t}}{\prod\limits_{\substack{k=1 \\ k\neq j}}^n (\lambda_k-\lambda_j)}=\frac{e^{-\lambda_n t}}{\prod\limits_{k=1}^{n-1} (\lambda_k-\lambda_n)}
\end{equation}
and noting (3.4), Proposition 4.1 (and (2.1)) therefore follows.
\end{proof}

\begin{rem}
With a little work and noting the Hermite-Genocchi integral relation (see Theorem 8.1 below), the density formula (2.1) can be re-written as an integral over the relevant simplex and must therefore be a divided difference.
\end{rem}

\begin{rem}
Propositions 2 and 3 of \cite{Baxter1997} show divided difference interpretations for more general forms involving exponentials than the one in (2.1).
\end{rem}

\begin{rem}
We will see below that expression (4.1), unlike expression (2.1), has a valid interpretation even in an instance when there are repeats of some or all of the $\lambda_i$'s.
\end{rem}

\section{Convolution of exponential random variables with identical parameters}
When $n$ independent exponential random variables $X_i$ have identical parameter $\lambda$, their sum $X_1 + \dots + X_n$ has the Erlang distribution with parameters $(n,\lambda)$. Its density, $Erl_{n,\lambda}(t)$, is defined by
\begin{equation}\notag
Erl_{n,\lambda}(t)=\frac{\lambda^n t^{n-1}}{(n-1)!}e^{-\lambda t}, \quad t \ge 0,
\end{equation}
see \cite{Akkouchi2008}, which is the gamma density with an integer shape parameter.

We have from (4.1) with $n=2$ and $\lambda_2=\lambda_1$
\begin{equation}\notag
S_2(t)=\lambda_1^2e[-\lambda_1,-\lambda_1].
\end{equation}
However, applying Definition 3.1 in this case would lead to division by zero. The extension to the instance of repeats in the arguments for the divided difference (sometimes called the \textit{confluent} or \textit{osculatory} case) is provided by its integral form.

Consider again the first-order divided difference $f[a_1,a_2]$
\begin{equation}\notag
f[a_1,a_2]=\frac{f(a_2)-f(a_1)}{a_2-a_1}=\frac{1}{a_2-a_1}\int_{a_1}^{a_2}f'(u)du,
\end{equation}
for an arbitrary variable $u$. Applying a change of variable to $v$ using $u=a_1+v(a_2-a_1)$ yields
\begin{equation}\notag
\begin{split}
f[a_1,a_2]&=\frac{1}{a_2-a_1} \int_{0}^{1}f'(a_1+v(a_2-a_1))(a_2-a_1)dv \\
&=\int_{0}^{1}f'(a_1+v(a_2-a_1))dv.
\end{split}
\end{equation}
It follows, when $a_2=a_1$ we have:
\begin{equation}\notag
f[a_1,a_1]=f'(a_1)\int_{0}^{1}dv = f'(a_1).
\end{equation}

\begin{thm}
(Hermite). 
Let $a_1,\dots,a_k$ be real (not necessarily distinct) and let $f(x)$ have a continuous $(k-1)$th derivative in the interval $[a_{min},a_{max}]$, where $a_{min}$ and $a_{max}$ are (respectively) the minimum and maximum of $a_1,\dots,a_k$. Then
\begin{equation}\notag
\begin{split}
f[a_1,\dots,a_k]= \int_0^1\int_0^{v_1} \dots \int_0^{v_{k-2}}f^{(k-1)}\big( a_1&+v_1(a_2-a_1)+v_2(a_3-a_2)+ \\
&\dots +v_{k-1}(a_k-a_{k-1})\big) dv_{k-1} \dots dv_1,
\end{split}
\end{equation}
where $f^{(m)}(a)$ denotes the $m$th-order derivative of $f(x)$ evaluated at $x=a$.
\end{thm}
\begin{proof}
See \cite{Schatzman2002} Theorem 4.2.3.
\end{proof}

As a consequence, the divided difference definition is extended to a (unique) continuous function of the points $a_1, \dots, a_k$ so long as the variables are evaluated within the interval of continuity of the $(k-1)$th derivative of $f(x)$. The instance of $k$ arguments $a_1, \dots, a_1$ is therefore found as:
\begin{equation}
\begin{split}
f[a_1,\dots,a_1]&= f^{(k-1)}(a_1)\int_0^1\int_0^{v_1} \dots \int_0^{v_{k-2}}dv_{k-1} \dots dv_1 \\
&= f^{(k-1)}(a_1)\int_0^1\int_0^{v_1} \dots \int_0^{v_{k-3}}v_{k-2}dv_{k-2} \dots dv_1 \\
&= f^{(k-1)}(a_1)\int_0^1\int_0^{v_1} \dots \int_0^{v_{k-4}}\frac{v_{k-3}^2}{2!}dv_{k-3} \dots dv_1 \\
&\dots \\
&= f^{(k-1)}(a_1)\frac{1}{(k-1)!}.
\end{split}
\end{equation}
Using (5.1), we may now extend the interpretation of (4.1) and state that the density for the sum of $n$ independent exponential random variables with identical parameter $\lambda$ is given by:
\begin{equation}\notag
S_n(t)=\lambda^n e[\underbrace{-\lambda,\dots,-\lambda}_\text{n times}] =\frac{\lambda^ne^{(n-1)}(-\lambda)}{(n-1)!} =\frac{\lambda^n t^{n-1}e^{-\lambda t}}{(n-1)!},
\end{equation}
the Erlang density with parameters $(n, \lambda)$.

\section{Convolution of exponential random variables in general - a novel representation}
In this section we continue to consider the density for sums of independent exponential random variables and develop an alternative representation for the general case of the sum of independent Erlang distributed random variables. Expressions for this density have been derived in a number of previous articles, such as \cite{Mathai1982}, \cite{Harrison1990}, \cite{Amari1997}, \cite{Coelho1998} and \cite{Jasiulewicz2003}. The technique used in these papers is either by taking Laplace transform of the convolution of random variables and inspection for its inverse, or through repeated integration by parts. Here, we provide a new direct approach by exploiting the representation of the hypo-exponential density in Proposition 4.1.

\subsection{Sums of exponentials and an Erlang distributed random variable}
In preparation, we examine the case studied in \cite{Khuong2006} where the density for the sum $Y=X_1 + \dots +X_n$ of $n$ independent exponential random variables is considered, with $m_1>1$ of these having parameter $\lambda_1$ and so their sum, $Z_1$, has the Erlang distribution with parameters $(m_1,\lambda_1)$. The remaining $n-m_1$, $X_i$'s, are independent exponential random variables with distinct parameters $\lambda_2, \dots, \lambda_{k}$, where $k:=n-m_1+1$. Looking to equation (4.1), the divided difference representation for $S_n(t)$, this case requires the determination of:
\begin{equation}
\begin{split}
S_n(t)&=\lambda_1^{m_1}\lambda_2 \dots \lambda_k e[\underbrace{-\lambda_1, \dots, -\lambda_1}_\text{$m_1$ times},-\lambda_2, \dots, -\lambda_k] \\
&=\lambda_1^{m_1}\lambda_2 \dots \lambda_k e[-\lambda_1^{(m_1)},-\lambda_2, \dots, -\lambda_{k}].
\end{split}
\end{equation}
In (6.1), $-\lambda_1^{(m_1)}$ denotes $m_1$ occurrences of $-\lambda_1$ in the argument list. One way to proceed is simply to apply the recurrence definition of the divided difference (3.3) across any two distinct arguments, say:
\begin{equation}
\begin{split}
e[-\lambda_1^{(m_1)},-\lambda_2, \dots, -\lambda_{k}]&=\frac{1}{\lambda_1-\lambda_k}\Big(e[-\lambda_1^{(m_1-1)},-\lambda_2, \dots, -\lambda_{k}] \\
& \qquad -e[-\lambda_1^{(m_1)},-\lambda_2, \dots, -\lambda_{k-1}]\Big)
\end{split}
\end{equation}
and to continue until all remaining distinct arguments are exhausted so that only repeats remain and then apply (5.1).

To clarify this procedure, take the illustrative example considered in \cite{Harrison1990}. Let $Y=Z_1+X_2$ where $Z_1$ has the density $Erl_{3,\lambda_1}(t)$ and $X_2$ is exponential with parameter $\lambda_2$. Then, beginning with (6.1), we may find the density for $Y$ using (6.2) as follows:
\begin{equation}\notag
\begin{split}
S_n(t)&=\lambda_1^3\lambda_2e[-\lambda_1^{(3)},-\lambda_2] \\
&=\lambda_1^3\lambda_2\Big(\frac{e[-\lambda_1^{(2)}, -\lambda_2]-e[-\lambda_1^{(3)}]}{\lambda_1-\lambda_2}\Big) \\
&=\lambda_1^3\lambda_2\Big(\frac{\big(e[-\lambda_1, -\lambda_2] - e[-\lambda_1^{(2)}]\big)}{(\lambda_1-\lambda_2)^2} -\frac{t^2e^{-\lambda_1t}}{2(\lambda_1-\lambda_2)}\Big) \quad \text{ (using (5.1))}\\
&=\lambda_1^3\lambda_2\Big(\frac{e^{-\lambda_2t}-e^{-\lambda_1t}}{(\lambda_1-\lambda_2)^3} - \frac{te^{-\lambda_1t}}{(\lambda_1-\lambda_2)^2}-\frac{t^2e^{-\lambda_1t}}{2(\lambda_1-\lambda_2)}\Big) \quad \text{ (using (5.1) again)} \\
&=\lambda_1^3\lambda_2\Big(\frac{e^{-\lambda_1t}}{(\lambda_2-\lambda_1)^3} - \frac{te^{-\lambda_1t}}{(\lambda_2-\lambda_1)^2} + \frac{t^2e^{-\lambda_1t}}{2(\lambda_2-\lambda_1)}+\frac{e^{-\lambda_2t}}{(\lambda_1-\lambda_2)^3}\Big)
\end{split}
\end{equation}
which agrees with \cite{Harrison1990}, p.76 (after correcting for a mistaken division by 2 in the $D_1$ term there). 

In general, the above approach, whilst perfectly correct, very quickly leads to a profusion of calculations and is only practical for a small number of repeats or distinct arguments.

\begin{pro}
Let $a_1,\dots,a_m$ be distinct and let $f(x)$ have a continuous $k_1$th derivative in the interval $[a_{min},a_{max}]$. Then
\begin{equation}\notag
\frac{\partial^{k_1}}{\partial a_1^{k_1}}f[a_1,a_2, \dots, a_m]=k_1!f[a_1^{(k_1+1)}, a_2,\dots,a_m],
\end{equation}
where the notation $a^{(i)}$ signifies that $a$ appears $i$ times in the list of arguments for the divided difference term.
\end{pro}
\begin{proof}
See Appendix A.1.
\end{proof}

We return to the case considered by \cite{Khuong2006} for general $k$. Using (4.1) and Proposition 6.1, we may therefore conclude the following proposition giving a concise expression for the density $f_Y(t)$ for $Y=Z+X_2 + \dots + X_k$:
\begin{pro}
Let $Z_1$ be a random variable having the Erlang distribution with parameters $(m_1,\lambda_1)$ and $X_i$, for $i=2, \dots,k$, be mutually independent exponential random variables with distinct parameters $\lambda_i$ and independent of $Z_1$. Then for $e(x)=e^{xt}$, the density $f_Y(t)$ for $Y=Z_1+X_2 + \dots + X_k$ is given by:
\begin{equation}\notag
\begin{split}
f_Y(t)&=\lambda_1^{m_1}\lambda_2\dots\lambda_k e[-\lambda_1^{(m_1)}, -\lambda_2, \dots, -\lambda_k] \\
&=\frac{\lambda_1^{m_1}\lambda_2\dots\lambda_k}{(m_1-1)!}.\frac{\partial^{m_1 - 1}}{\partial (-\lambda_1)^{m_1-1}}e[-\lambda_1, -\lambda_2 \dots, -\lambda_k].
\end{split}
\end{equation}
\end{pro}
\noindent
A closed-form formula for the density of the sum $Y$ may now be derived as follows:
\begin{equation}
\begin{split}
f_Y(t)&=\frac{\lambda_1^{m_1}\lambda_2\dots\lambda_k}{(m_1-1)!}.\frac{\partial^{m_1 - 1}}{\partial (-\lambda_1)^{m_1-1}}e[-\lambda_1, -\lambda_2 \dots, -\lambda_k] \\
&=\frac{\lambda_1^{m_1}\lambda_2\dots\lambda_k}{(m_1-1)!}.\frac{\partial^{m_1 - 1}}{\partial (-\lambda_1)^{m_1-1}}\Big(\frac{e^{-\lambda_1t}}{\prod\limits_{q=2}^k (\lambda_q-\lambda_1)}+\sum_{j=2}^k\frac{e^{-\lambda_jt}}{\prod\limits_{\substack{q=1 \\ q\neq j}}^k (\lambda_q-\lambda_j)}\Big) \\
&=\frac{\lambda_1^{m_1}\lambda_2\dots\lambda_k}{(m_1-1)!}\Big(\sum\limits_{\substack{\sum_{j=1}^k r_j=m_1-1\\r_i \ge 0}}\frac{(m_1-1)!}{r_1! \dots r_k!}e^{-\lambda_1t}t^{r_1}\prod\limits_{q=2}^k\frac{(-1)^{r_q}r_q!}{(\lambda_q-\lambda_1)^{r_q+1}} + \sum_{j=2}^k\frac{e^{-\lambda_jt}(m_1-1)!}{(\lambda_1-\lambda_j)^{m_1}\prod\limits_{\substack{q=2 \\ q\neq j}}^k(\lambda_q-\lambda_j)}\Big) \\
& \quad \text{ (using the general Leibniz rule, see Appendix A.2)} \\
&=\lambda_1^{m_1}\lambda_2\dots\lambda_k\Big(e^{-\lambda_1t}(-1)^{m_1-1}\sum\limits_{\substack{\sum_{j=1}^k r_j=m_1-1\\r_i \ge 0}}\frac{(-t)^{r_1}}{r_1!\prod\limits_{q=2}^k(\lambda_q-\lambda_1)^{r_q+1}} + \sum_{j=2}^k\frac{e^{-\lambda_jt}}{(\lambda_1-\lambda_j)^{m_1}\prod\limits_{\substack{q=2 \\ q\neq j}}^k(\lambda_q-\lambda_j)}\Big),
\end{split}
\end{equation}
for $t \ge 0$. 

For the case $k=2$ and general $m_1$, the density may also be found in the following succinct manner. Set $r=m_1-1$. Then using Proposition 6.2,
\begin{equation}
\begin{split}
f_Y(t)&=\frac{\lambda_1^{r+1}\lambda_2}{r!}\frac{d^r}{d(-\lambda_1)^r}e[-\lambda_1, -\lambda_2]\\
&=\frac{\lambda_1^{r+1}\lambda_2e^{-\lambda_2t}}{r!}\frac{d^r}{dx^r}\Big(\frac{t(1-e^{-x})}{x}\Big)\Big(\frac{dx}{d(-\lambda_1)}\Big)^r, \text{ for } x=(\lambda_1-\lambda_2)t\\
&=\frac{\lambda_1^{r+1}\lambda_2e^{-\lambda_2t}} {r!}\frac{d^r}{dx^r}\Big(\frac{t\gamma(1,x)}{x}\Big)(-t)^r=\frac{\lambda_1^{r+1}\lambda_2e^{-\lambda_2t}} {r!}\frac{\gamma(r+1,x)}{x^{r+1}}t^{r+1} \quad \text{ (using (3.8)) }\\
&=\frac{\lambda_1^{m_1}\lambda_2e^{-\lambda_2t}\gamma(m_1,(\lambda_1-\lambda_2)t)}{(m_1-1)!(\lambda_1-\lambda_2)^{m_1}}, \quad \text{ for } t > 0
\end{split}
\end{equation}
and $\lambda_1>\lambda_2$. For the instance $\lambda_2>\lambda_1$, replace $\gamma(m_1,(\lambda_1-\lambda_2)t)$ by a corresponding confluent hypergeometric representation, (3.7) above. The equivalence with (6.3) is seen once we rearrange terms there (for $k=2$) to find
\begin{equation}\notag
\begin{split}
f_Y(t)&=\lambda_1^{m_1}\lambda_2\Big(\frac{e^{-\lambda_2t}}{(\lambda_1-\lambda_2)^{m_1}} -e^{-\lambda_1t}\sum_{r=0}^{m_1-1}\frac{t^r}{r!(\lambda_1-\lambda_2)^{m_1-r}} \Big) \\
&=\lambda_1^{m_1}\lambda_2\frac{e^{-\lambda_2t}}{(\lambda_1-\lambda_2)^{m_1}}\Big(1 -e^{-(\lambda_1-\lambda_2)t}\sum_{r=0}^{m_1-1}\frac{[(\lambda_1-\lambda_2)t]^r}{r!} \Big) 
\end{split}
\end{equation}
and (6.4) follows after applying the finite series representation for $\gamma(n,z)$.

\subsection{Sums of exponentials in general}
Returning to Proposition 6.1, we may further take partial derivatives w.r.t. $a_2$ and write:
\begin{equation}
\begin{split}
\frac{\partial^{k_1+k_2}}{\partial a_1^{k_1}\partial a_2^{k_2}}f[a_1,a_2, \dots, a_m]&=k_1!\frac{\partial^{k_2}}{\partial a_2^{k_2}}f[a_1^{(k_1+1)}, a_2,\dots,a_m] \\
&=k_1!k_2!f[a_1^{(k_1+1)}, a_2^{(k_2+1)},\dots,a_m].
\end{split}
\end{equation}
The second equality in (6.5) follows by applying a similar reasoning used in the proof of Proposition 6.1. Repeating this for all the arguments, we state the following corollary:

\begin{cor}
Let $a_1,\dots,a_m$ be distinct and let $f(x)$ have a continuous $q$th derivative in the interval $[a_{min},a_{max}]$. Then, for integers $1<k_i \le q$, $i=1, \dots, m$,
\begin{equation}\notag
\frac{\partial^{k_1+\dots +k_m}}{\partial a_1^{k_1} \dots \partial a_m^{k_m}}f[a_1, \dots, a_m] = k_1! \dots k_m!f[a_1^{(k_1+1)}, \dots,a_m^{(k_m+1)}]. 
\end{equation}
\end{cor}
\noindent
\newline
The interested reader can look further to \cite{Ostrowski1966} Ch.1 or to the exercises and hints in \cite{Schatzman2002} Ch.4. 

Let $X_1, \dots, X_n$ be $n$ independent random variables having exponential distributions with parameters from the set \{$\lambda_i; i = 1, \dots, k \le n$\}, with $\lambda_i \ne \lambda_j$ when $i \ne j$. Suppose further that $m_i$ is the number of such random variables having parameter $\lambda_i$ so that $m_1+ \dots +m_k = n$. Then $Y=X_1+ \dots +X_n$ is the sum of $k$ independent random variables having the Erlang distribution with parameter $(m_i, \lambda_i)$ for $i=1, \dots, k$. Applying Corollary 6.1, the following result is then immediately apparent:
\begin{pro}
(A General Representation). Let $Y=Z_1+ \dots +Z_k$ be the sum of $k$ independent random variables having Erlang distributions with parameters, respectively, $(m_i, \lambda_i)$ for $i=1, \dots, k$. For $e(x)=e^{xt}$, the density for $Y$, $f_Y(t)$, for $t \ge 0$ is given by 
\begin{equation}
\begin{split}
f_Y(t)&=\Big(\prod_{i=1}^k \lambda_i^{m_i}\Big) e[-\lambda_1^{(m_1)}, \dots, -\lambda_k^{(m_k)}] \\
&=\prod_{i=1}^k \frac{\lambda_i^{m_i}}{(m_i-1)!}.\frac{\partial^{n - k}}{\partial (-\lambda_1)^{m_1-1} \dots \partial (-\lambda_k)^{m_k-1}}e[-\lambda_1, \dots, -\lambda_k],
\end{split}
\end{equation}
where $n= m_1+ \dots + m_k$.
\end{pro}
\begin{proof}
The correctness of this representation will also be demonstrated via Lemma 7.1 below.
\end{proof}
\noindent
This particular representation for the sum of independent Erlang distributed random variables appears to be novel. A closed-form expression for (6.6) is easily found with Lemma 6.1.
\begin{lem}
Let $a_1,\dots,a_m$ be distinct and let $f(x)$ have a continuous $q$th derivative in the interval $[a_{min},a_{max}]$. Then, for integers $1<k_i \le q$, $i=1, \dots, m$, and $k=k_1+ \dots +k_m$:
\begin{equation}\notag
\frac{\partial^k}{\partial a_1^{k_1} \dots \partial a_m^{k_m}}f[a_1, \dots, a_m] = \sum_{i=1}^mk_i!\sum\limits_{\substack{\sum_{j=1}^m r_j=k_i\\r_j \ge 0}}\frac{f^{(r_i)}(a_i)(-1)^{k_i-r_i}}{r_i!}\prod\limits_{\substack{q=1 \\ q\neq i}}^m\frac{(k_q+r_q)!}{(a_i-a_q)^{k_q+r_q+1}r_q!}.
\end{equation}
\end{lem}
\begin{proof}
See Appendix A.2. 
\end{proof}
\noindent
We may now state the following corollary:
\begin{cor}
Let $Y=Z_1+ \dots +Z_k$ be the sum of $k$ independent random variables having the Erlang distribution with parameter $(m_i, \lambda_i)$ for $i=1, \dots, k$ and $\lambda_i \ne \lambda_j$ for $i \ne j$. Then the density for $Y$, $f_Y(t)$, for $t \ge 0$ is given by
\begin{equation}
f_Y(t)=\prod_{i=1}^k \lambda_i^{m_i}\Big\{\sum_{i=1}^k\frac{e^{-\lambda_it}(-1)^{m_i-1}}{\prod\limits_{\substack{j=1 \\ j\neq i}}^k (m_j-1)!} \times \sum\limits_{\substack{\sum_{j=1}^k r_j=m_i-1\\r_j \ge 0}}\frac{(-t)^{r_i}}{r_i!}\prod\limits_{\substack{q=1 \\ q\neq i}}^k\frac{(m_q+r_q-1)!}{(\lambda_q-\lambda_i)^{m_q+r_q}r_q!}\Big\}.
\end{equation}
\end{cor}
\begin{proof}
The result is immediate after applying Lemma 6.1 to equation (6.6) and rearranging terms.
\end{proof}
\noindent
Corollary 6.2 agrees with Theorem 1 of \cite{Harrison1990}. A further manipulation of (6.7) gives:
\begin{equation}\notag
\begin{split}
f_Y(t)=&\sum_{i=1}^k\lambda_i^{m_i}e^{-\lambda_it}\sum_{n=1}^{m_i}\frac{(-1)^{m_i-n}t^{n-1}}{(n-1)!} \\
&\quad \times \sum\limits_{\substack{\sum_{j=1}^k r_j=m_i-n\\j \ne i, r_j \ge 0}}\prod\limits_{\substack{q=1 \\ q\neq i}}^k \binom{m_q+r_q-1)}{r_q}\frac{\lambda_q^{m_q}}{(\lambda_q-\lambda_i)^{m_q+r_q}},
\end{split}
\end{equation}
agreeing with Theorem 1 of \cite{Jasiulewicz2003}. 

Consider the case $k=2$, where $Z_1$ and $Z_2$ are independent variables and have the densities $Erl_{m_1,\lambda_1}(t)$ and $Erl_{m_2,\lambda_2}(t)$, respectively, with $\lambda_1 \ne \lambda_2$. Using (6.6), we may determine the density for $Y=Z_1+Z_2$ and generalise (6.4) directly as follows:
\begin{equation}
\begin{split}
f_Y(t)&=\lambda_1^{m_1}\lambda_2^{m_2} e[-\lambda_1^{(m_1)}, -\lambda_2^{(m_2)}] \\
&=\frac{\lambda_1^{m_1}\lambda_2^{m_2}}{(m_1-1)!(m_2-1)!}.\frac{\partial^{m_1+m_2 - 2}}{\partial (-\lambda_1)^{m_1-1} \partial (-\lambda_2)^{m_2-1}}e[-\lambda_1, -\lambda_2] \\
&=\frac{\lambda_1^{m_1}\lambda_2^{m_2}}{(m_1-1)!(m_2-1)!}\frac{\partial^{m_2 - 1}}{\partial (-\lambda_2)^{m_2-1}}\Big(\frac{e^{-\lambda_2t}\gamma(m_1,(\lambda_1-\lambda_2)t)}{(\lambda_1-\lambda_2)^{m_1}}\Big) \quad \text{ (following the steps to (6.4))} \\
&=\frac{\lambda_1^{m_1}\lambda_2^{m_2}t^{m_1}}{m_1!(m_2-1)!}\frac{\partial^{m_2 - 1}}{\partial (-\lambda_2)^{m_2-1}}\Big(e^{-\lambda_1t}M(1,m_1+1,(\lambda_1-\lambda_2)t)\Big) \quad \text{ (using (3.7))} \\
&=\frac{\lambda_1^{m_1}\lambda_2^{m_2}t^{m_1+m_2-1}e^{-\lambda_1t}}{\Gamma(m_1+m_2)}M(m_2, m_1+ m_2, (\lambda_1-\lambda_2)t), \quad \text{ (using (3.9))}\\
&=\frac{\lambda_1^{m_1}\lambda_2^{m_2}t^{m_1+m_2-1}e^{-\lambda_2t}}{\Gamma(m_1+m_2)}M(m_1, m_1+ m_2, (\lambda_2-\lambda_1)t), \quad t>0.
\end{split}
\end{equation}
In the final step we applied Kummer's transformation.

\section{A representation for the moment generating function for sums of Erlangs}
For a continuous random variable $Y$ with density function $f_Y(t)$ for $t \ge 0$, its moment generating function (m.g.f.), $M_Y(s)$, is defined by:
\begin{equation}
M_Y(s)=\mathbb{E}(e^{st})=\int_0^\infty e^{st}f_Y(t)dt.
\end{equation}
If the m.g.f. exists then it uniquely determines the distribution. The m.g.f. for a random variable with density $Erl_{m,\lambda}(t)$ is $\lambda^m/(\lambda-s)^m$ (see \cite{Ross1997} p.65) and so when $Y$ is the sum of $k$ independent Erlang random variables with densities $Erl_{m_i,\lambda_i}(t)$, for $i=1, \dots, k$, we have:
\begin{equation}\notag
M_Y(s)= \prod_{i=1}^k \Big(\frac{\lambda_i}{\lambda_i-s}\Big)^{m_i}.
\end{equation}
\begin{lem}
Let $Y=Z_1+\dots +Z_k$ where $Z_i$'s are independent Erlang distributed random variables with densities $Erl_{m_i,\lambda_i}(t)$, $i=1, \dots, k$, respectively and $\lambda_i \ne \lambda_j$ for $i \ne j$. Then the m.g.f. of $Y$ can be characterised by:
\begin{equation}
M_Y(s)=\int_0^\infty e^{st}\prod_{i=1}^k \frac{\lambda_i^{m_i}}{(m_i-1)!}.\frac{\partial^{m - k}}{\partial (-\lambda_1)^{m_1-1} \dots \partial (-\lambda_k)^{m_k-1}}e[-\lambda_1, \dots, -\lambda_k]dt
\end{equation}
where $m=m_1+ \dots+m_k$ and divided differences are taken over the function $e(x)=e^{xt}$.
\end{lem}
We provide a proof of this lemma and in doing so, once more confirm the correctness of Proposition 6.3. The proof is given as it will also form the basis for a further development in the next section.
\begin{proof}
Applying the Leibniz integral rule to (7.2) and using the Lagrange polynomial representation for the divided difference term $e[-\lambda_1, \dots, -\lambda_k]$, we have:
\begin{equation}
\begin{split}
M_Y(s)&=\prod_{i=1}^k \frac{\lambda_i^{m_i}}{(m_i-1)!}.\frac{\partial^{m - k}}{\partial (-\lambda_1)^{m_1-1} \dots \partial (-\lambda_k)^{m_k-1}}\sum_{j=1}^k\frac{\int_0^\infty e^{-(\lambda_j-s)t}dt}{\prod\limits_{\substack{q=1 \\ q\neq j}}^k (\lambda_q-\lambda_j)} \\
&=\prod_{i=1}^k \frac{\lambda_i^{m_i}}{(m_i-1)!}.\frac{\partial^{m - k}}{\partial (-\lambda_1)^{m_1-1} \dots \partial (-\lambda_k)^{m_k-1}}\sum_{j=1}^k\frac{1}{(\lambda_j-s)\prod\limits_{\substack{q=1 \\ q\neq j}}^k (\lambda_q-\lambda_j)} \\
&=\prod_{i=1}^k \frac{\lambda_i^{m_i}}{(m_i-1)!}.\frac{\partial^{m - k}}{\partial (-\lambda_1)^{m_1-1} \dots \partial (-\lambda_k)^{m_k-1}} \prod_{j=1}^k\frac{1}{(\lambda_j-s)} \quad \text{ (from (3.6))} \\
&=\prod_{i=1}^k \frac{\lambda_i^{m_i}}{(m_i-1)!}. \prod_{j=1}^k \frac{d^{m_j - 1}}{d(-\lambda_j)^{m_j-1}} \frac{1}{(\lambda_j-s)}=\prod_{i=1}^k \frac{\lambda_i^{m_i}}{(m_i-1)!}. \prod_{j=1}^k\frac{(m_j-1)!}{(\lambda_i-s)^{m_j}}\\ &=\prod_{i=1}^k \Big(\frac{\lambda_i}{\lambda_i-s}\Big)^{m_i}
\end{split}
\end{equation}
as required.
\end{proof}

\section{Fractional calculus and the extension to sums of independent gamma distributed random variables}
A gamma distributed random variable $Z$ with parameter $(\alpha, \beta)$ and mean $(\alpha/ \beta)$ has density and m.g.f. given, respectively, by:
\begin{equation}\notag
G_{\alpha,\beta}(t)=\frac{\beta^{\alpha} t^{\alpha-1}e^{-\beta t}}{\Gamma(\alpha)} \text{ and } M_Z(s)=\Big(\frac{\beta}{\beta-s}\Big)^\alpha \quad \alpha, \beta>0.
\end{equation}
By correspondence, it is tempting, but wrong, to extend of (6.6) to gamma variables simply by replacing $m_i$, $(m_i-1)!$ and $\lambda_i$ by (respectively) $\alpha_i$, $\Gamma(\alpha_i)$ and $\beta_i$. This would then extend Proposition 6.3 to include a representation for the density for the sum of $k$ independent gamma random variables with parameters $(\alpha_i, \beta_i)$. Such a proposal would appear to be reasonable given that in Sec. 5 we saw that the Erlang density with parameters $(n, \lambda)$ may be found as:
\begin{equation}
Erl_{n,\lambda}(t) = \frac{\lambda^n}{(n-1)!} \Big(\frac{d^{n-1}}{d(-\lambda)^{n-1}}e^{-\lambda t}\Big)=\frac{\lambda^n t^{n-1}e^{-\lambda t}}{(n-1)!}
\end{equation}
and if, for $f(x)=e^{xt}$, $\frac{d^v}{dx^v}f(x) = t^ve^{xt}$ were true for noninteger $v > 0$ then we could similarly conclude:
\begin{equation}
G_{\alpha,\beta}(t) = \frac{\beta^{\alpha}}{\Gamma(\alpha)} \Big(\frac{d^{\alpha-1}}{d(-\beta)^{\alpha-1}}e^{-\beta t}\Big)=\frac{\beta^{\alpha} t^{\alpha-1}e^{-\beta t}}{\Gamma(\alpha)}.
\end{equation}
However, unlike the Erlang distribution, $\alpha$ is not restricted to being a positive integer and whilst $\beta^\alpha$ and $\Gamma(\alpha)$ may be obvious generalisations of $\beta^n$ and $(n-1)!$ for positive noninteger parameters, the derivative term in (8.2) needs elaborating and we will require the tools of \textit{fractional calculus} in order to formalise this extension.

The history of fractional calculus is almost as old as that of the calculus itself. However, compared with integer calculus, noninteger calculus ideas and methods are relatively unfamiliar. Its development has been comparatively slower and a reflection of this is that an account providing a systematic treatment of the subject did not appear until the publication in 1974 of the book by Oldham and Spanier \cite{Oldham1974}. The interested reader can look to \cite{Miller1993} and to \cite{Podlubny1999} for two further accessible textbooks on the subject. We will draw on these and other sources but present only the necessary definitions and results required to complete our discussion.

The most widely investigated and used definition of the fractional derivative is the \textit{Riemann-Liouville} (RL) definition (sometimes referred to as the \textit{Abel-Riemann} definition). 

Let $x \in \mathbb{R}$. For a function $f \in L^1[a,b]$, $-\infty < a<b<+\infty$, the (left-sided) RL fractional integral of order $\nu>0$ is defined as
\begin{equation}
_aI_x^{\nu}f(x)=\frac{1}{\Gamma(\nu)}\int_a^x(x-\tau)^{\nu-1}f(\tau)d\tau, \quad \text {for } x \in [a,b],
\end{equation}
where, $L^1[a,b]$ denotes the set of Lebesgue integrable functions on $[a,b]$. For completeness, ${_a}I_x^0f(x)=f(x)$. The fractional integral operator has the linearity property $_aI_x^{p}\big(bf(x)+cg(x)\big)$ = $b\big({_a}I_x^{p}f(x)\big)$+ $c\big({_a}I_x^{p}g(x)\big)$ for $b$, $c$ constants and the semigroup property $_aI_x^{p}(_aI_x^{q})$= $_aI_x^{p+q}$.

The (left-sided) RL fractional derivative of order $\nu >0$ is defined by:
\begin{equation}
\begin{split}
_aD_x^\nu f(x)&=\frac{d^m}{dx^m}\big({_a}I_x^{m-\nu} f(x)\big) \\
&=
\begin{cases}
\frac{1}{\Gamma(m-\nu)}\frac{d^m}{dx^m}\int_a^x(x-\tau)^{m-\nu-1}f(\tau)d\tau, \quad m-1< \nu < m\\
\frac{d^m}{dx^m}f(x), \quad \nu=m \in \mathbb{N}.
\end{cases}
\end{split}
\end{equation}
The definition (8.4) defines the fractional derivative of a noninteger order $\nu>0$ as a composition of fractional integration of order $m-\nu$ followed by differentiation of integer order $m$ where $m$ is the smallest integer greater than $\nu$. From the definition, $_aD_x^{\nu}\big({_aI_x^{\nu}}f(x)\big)$=$_aD_x^{m}{_aI_x^{m-\nu}}\big({_aI_x^{\nu}}f(x)\big)$ = $_aD_x^{m}\big({_aI_x^{m}}f(x)\big)$=$f(x)$ and hence $_aD_x^{\nu}$ is a left-inverse to $_aI_x^{\nu}$. However, $_aI_x^{\nu}(_aD_x^{\nu}f(x))=f(x)$ is only true when $f^{(\nu-j)}(a)=0$ for $j=1,\dots, m$, where $m-1 < \nu \le m$. When the order $\nu$ is a positive integer, $_aD_x^\nu f(x)$ is the conventional integer-order derivative. It is easily shown that the fractional derivative conforms with the linear transformation property: 
\begin{equation}\notag
_{a}D_{x}^{\nu}f(bx+c)=b^{\nu}\big({_{ab+c}}D_{y}^{\nu}f(y)\big), \quad y=bx+c, b>0.
\end{equation}
Note the inclusion of the upper and lower terminals in the notation and definitions. It can be seen that the fractional integral is always nonlocal (i.e. dependent on $a$, the lower terminal) and the fractional derivative is generally nonlocal unless it is of an integer order. For a thorough discussion see Ch.5 in \cite{Oldham1974}, alternatively \cite{Podlubny1999} Ch.2 or \cite{Miller1993} Ch.4.

We now examine the validity of (8.2) by determining $_aD_x^v f(x)$ for $f(x)=e^{xt}$ for noninteger $v>0$. Let $m-1 < v < m$ so that $\xi= m-v>0$. We have then
\begin{equation}\notag
\begin{split}
_aD_x^v e^{xt}&=\frac{d^m}{dx^m}(_aI_x^{\xi} e^{xt})=\frac{1}{\Gamma(\xi)}\frac{d^m}{dx^m}\int_a^x(x-\tau)^{\xi-1}e^{t \tau}d\tau \\
&=\frac{1}{\Gamma(\xi)}\frac{d^m}{dx^m}\frac{e^{xt}}{t^{\xi}}\int_0^{(x-a)t}u^{\xi-1}e^{-u}du \quad \text{ ($\tau \mapsto x-u/t)$} \\
&=\frac{1}{\Gamma(\xi)}\frac{d^m}{dx^m}\frac{e^{xt}}{t^{\xi}}\gamma(\xi,(x-a)t).
\end{split}
\end{equation}
Taking the lower terminal as $a =-\infty$, the \textit{Liouville} form of the RL fractional derivative, $\gamma(\xi,(x-a)t)$ $\to$ $\Gamma(\xi)$ so yielding:
\begin{equation}\notag
_{-\infty}D_x^v e^{xt}=\frac{d^m}{dx^m}\frac{e^{xt}}{t^{\xi}} = t^{m-\xi}e^{xt}=t^v e^{xt}.
\end{equation}
Hence (8.2) has a meaningful correspondence with (8.1) under the Liouville definition for fractional derivatives (sometimes also referred to as the \textit{Liouville-Weyl} definition). 

With the Liouville definition, a sufficient condition that (8.3) converge is that $f(-x)=O(x^{-\nu -\epsilon})$, $\epsilon>0$, $x \to \infty$. Integrable functions satisfying this property are sometimes referred to as functions of \textit{Liouville class}. It is straightforward to verify that $f(x)=e^{cx}$ (with $c>0$) and $f(x) = x^{-c}$ (with $0<v<c$) and $x<0$ are of Liouville class. (See \cite{Miller1993} for further discussion). 

Partial and mixed fractional derivatives under the Riemann-Liouville definition are also possible (see Sec. 24 of \cite{Samko1993} and \cite{Li2011}). We adapt Definition 3.1 in \cite{Li2011} to the Liouville form:
\begin{dfn}
The (mixed) partial Liouville fractional derivative with order $\nu=\sum_{i=1}^k \nu_i$ $(\nu_i$th order in $x_i$ direction, $i=1, \dots, k)$ is defined as follows:
\begin{equation}\notag
\begin{split}
&\frac{\partial^\nu}{\partial x_1^{\nu_1} \dots \partial x_k^{\nu_k}}g(x_1, \dots, x_k) := {_{-\infty}}D_{x_1}^{\nu_1} \dots {_{-\infty}}D_{x_k}^{\nu_k}g(x_1, \dots, x_k) \\
&=\quad \prod_{i=1}^k\frac{1}{\Gamma(m_i-\nu_i)}.\frac{\partial^m}{\partial x_1^{m_1} \dots \partial x_k^{m_k}}\int_{-\infty}^{x_1}\dots\int_{-\infty}^{x_k}\prod_{i=1}^k(x_i-\xi_i)^{\eta_i}g(\xi_1, \dots, \xi_k) d\xi_k \dots d\xi_1,
\end{split}
\end{equation}
where $m=\sum_{i=1}^k m_i$, $\eta_i = (m_i-\nu_i-1)$, $m_{i-1}< \nu_i < m_i$, and $m_i \in \mathbb{Z^+}$ for $i=1, \dots, k$.
\end{dfn}
Generalising (8.4), the mixed fractional derivative is a composition of mixed fractional integration followed by integer-order mixed partial differentiation. Note that when $g(.)$ is a completely multiplicatively separable function of its variables then the mixed partial derivative becomes an integer-order mixed partial derivative of a product of Liouville fractional integrals for each variable.

\subsection{A representation for the density for sums of independent gamma random variables}
The following lemma will assist us further in our extension to sums of independent gamma random variables.
\begin{lem}
Let $f(x)=(x+b)^{-c}$ for $x \in \mathbb{R}$ with $b$ and $c$ constants with $c \ge 1$. Then: 

(i) For $0<\nu<1$:
\begin{equation}\notag
_{-\infty}I_x^{\nu}f(x)=(-1)^{\nu}\frac{\Gamma(c-\nu)}{\Gamma(c)}(x+b)^{-(c-\nu)},
\end{equation}
where $(-1)^{\nu}$ is a complex coefficient. 

(ii) For any $0 \le m-1 < \nu < m$ with $m \in \mathbb{N}$:
\begin{equation}\notag
_{-\infty}D_x^{\nu}f(x)=(-1)^{-\nu}\frac{\Gamma(c+\nu)}{\Gamma(c)}(x+b)^{-(c+\nu)},
\end{equation}
with $(-1)^{-\nu}$ complex when $\nu$ is noninteger.
\end{lem}

\begin{proof}
(i) We note that with $0 < \nu < 1$, $\nu<c$. The Liouville fractional integral (definition (8.3) with $a=-\infty$) gives:
\begin{equation}\notag
\begin{split}
_{-\infty}I_x^{\nu}f(x)&=\frac{1}{\Gamma(\nu)}\int_{-\infty}^x (x-\tau)^{\nu-1}(\tau+b)^{-c}d\tau \\
&=\frac{(x+b)^{\nu-c}(-1)^\nu}{\Gamma(\nu)}\int_0^1 (1-u)^{c-\nu-1}u^{\nu-1}du \quad \text{ }(\tau \mapsto (x+bu)/(1-u) ) \\
&=\frac{(x+b)^{\nu-c}(-1)^\nu}{\Gamma(\nu)}B(c-\nu, \nu) \\
&=\frac{(-1)^{\nu}\Gamma(c-\nu)}{\Gamma(c)}(x+b)^{-(c-\nu)},
\end{split}
\end{equation}
using the definition and properties of the beta function. 

(ii) We note, with $\nu > 0$ and $m-1< \nu < m$, $(m-\nu)<c$. The Liouville fractional derivative (definition (8.4) with $a=-\infty$) gives:
\begin{equation}\notag
_{-\infty}D_x^{\nu}f(x)=\frac{d^m}{dx^m}\big(_{-\infty}I_x^{m-\nu} f(x)\big) =\frac{(-1)^{m-\nu}\Gamma(c-m+\nu)}{\Gamma(c)}\frac{d^m}{dx^m}(x+b)^{-(c-m+\nu)},
\end{equation}
from (i) above. The $m$th order derivative of $(x+b)^{-(c-m+\nu)}$ w.r.t. $x$ can be expressed as:
\begin{equation}\notag
		\frac{d^{m}}{dx^{m}}(x+b)^{-(c-m+\nu)} =\frac{(-1)^{-m}\Gamma(c+\nu)}{\Gamma(c-m+\nu)}(x+b)^{-(c+\nu)}
\end{equation}
and the lemma follows after substitution and cancelling.
\end{proof}

\begin{pro}
Let $Z_1, \dots, Z_k$ be $k$ independent random variables with $Z_i$ having the gamma density $G_{\alpha_i, \beta_i}(t), i = 1, \dots, k$ and $\beta_i \ne \beta_j$ for $i \ne j$. Let $e[-\beta_1,\ldots,-\beta_k]$ be the divided difference for $e(x)=e^{xt}$ at points $-\beta_i$ $(i=1, \dots, k)$ then, at least for the Liouville definition of fractional derivatives, we may say that $Y=Z_1+ \dots +Z_k$ has density, $f_Y(t)$, given by:
\begin{equation}
f_Y(t)=\prod_{i=1}^k \frac{\beta_i^{\alpha_i}}{\Gamma(\alpha_i)}.{_{-\infty}}D_{-\beta_1}^{\alpha_1-1} \dots{ _{-\infty}}D_{-\beta_k}^{\alpha_k-1}e[-\beta_1, \dots, -\beta_k].
\end{equation}
\end{pro}
\begin{proof}
Let $f_Y(t)$ denote the density function for the sum $Y$. Substituting $f_Y(t)$ from (8.5) into (7.1), the m.g.f. for $Y$ is then expressed as:
\begin{equation}\notag
\begin{split}
M_Y(s)&=\int_0^\infty e^{st}\prod_{i=1}^k \frac{\beta_i^{\alpha_i}}{\Gamma(\alpha_i)}.{_{-\infty}}D_{-\beta_1}^{\alpha_1-1} \dots{ _{-\infty}}D_{-\beta_k}^{\alpha_k-1}e[-\beta_1, \dots, -\beta_k]dt \\
&=\prod_{i=1}^k \frac{\beta_i^{\alpha_i}}{\Gamma(\alpha_i)}.{_{-\infty}}D_{-\beta_1}^{\alpha_1-1} \dots{ _{-\infty}}D_{-\beta_k}^{\alpha_k-1}\sum_{j=1}^k\frac{\int_0^\infty e^{-(\beta_j-s)t}dt}{\prod\limits_{\substack{q=1 \\ q\neq j}}^k (\beta_q-\beta_j)},
\end{split}
\end{equation}
as the exchange of order of integrals is clearly permitted. The proof continues by following the same steps as that for Lemma 7.1 except at the penultimate line of (7.3) we apply the Liouville fractional derivative definition to give instead:
\begin{equation}
M_Y(s)=\prod_{i=1}^k \frac{\beta_i^{\alpha_i}}{\Gamma(\alpha_i)}.\prod_{j=1}^k {_{-\infty}D_{-\beta_j}^{\alpha_j-1}}\Big(\frac{1}{\beta_j-s}\Big) 
\end{equation}
where, for $\alpha_j < 1$, we note that ${_{-\infty}D_{-\beta_j}^{\alpha_j-1}}$ must be interpreted as a fractional integral, i.e.
\begin{equation}\notag
{_{-\infty}D_{-\beta_j}^{\alpha_j-1}} \equiv {_{-\infty}I_{-\beta_j}^{1-\alpha_j}}, \quad \alpha_j<1.
\end{equation}
Setting $c=1$ in Lemma 8.1, we have firstly, for $0<\alpha_j<1$ using part (i) of the lemma:
\begin{equation}\notag
_{-\infty}I_x^{1-\alpha_j}\Big(\frac{1}{x+s}\Big)=(-1)^{1-\alpha_j}\Gamma(\alpha_j)(x+s)^{-\alpha_j}
\end{equation}
or, multiplying both sides by $(-1)$, (using the linearity property for $_{a}I_x^{\nu}$)
\begin{equation}\notag
_{-\infty}I_x^{1-\alpha_j}\Big(\frac{1}{-x-s}\Big)=\Gamma(\alpha_j)(-x-s)^{-\alpha_j}.
\end{equation}
It follows, for $\beta_j>s$, we can write
\begin{equation}\notag
_{-\infty}I_{-\beta_j}^{1-\alpha_j}\Big(\frac{1}{\beta_j-s}\Big)=\frac{\Gamma(\alpha_j)}{(\beta_j-s)^{\alpha_j}}.
\end{equation}
Secondly, for $\alpha_j \ge 1$, using Lemma 8.1 part (ii):
\begin{equation}\notag
_{-\infty}D_x^{\alpha_j-1}\Big(\frac{1}{x+s}\Big)=(-1)^{1-\alpha_j}\Gamma(\alpha_j)(x+s)^{-\alpha_j}.
\end{equation}
Again, it follows, for $\beta_j>s$, we can write
\begin{equation}\notag
_{-\infty}D_{-\beta_j}^{\alpha_j-1}\Big(\frac{1}{\beta_j-s}\Big)=\frac{\Gamma(\alpha_j)}{(\beta_j-s)^{\alpha_j}}.
\end{equation}
Hence, (8.6) becomes
\begin{equation}\notag
M_Y(s)=\prod_{i=1}^k \Big(\frac{\beta_i}{\beta_i-s}\Big)^{\alpha_i},
\end{equation}
which we know is the m.g.f. for $Y$. As the moment generating function is unique to the density function, the proof is completed.
\end{proof}
Proposition 8.1 extends Proposition 6.3 to independent gamma distributed random variables and for the instance of $k=1$, (8.2) is therefore shown to be a valid statement under this proposition.

\begin{rem}
Observe that for values $\{-\beta_1, \dots, -\beta_k\}$ and constant $t$, the divided differences over the functions $e(x)=e^{xt}$ and $\text{exp}(x)=e^x$ satisfy: 
\begin{equation}\notag
e[-\beta_1, \dots, -\beta_k]=t^{k-1}\text{exp}[-\beta_1t, \dots, -\beta_kt],
\end{equation}
where $\text{exp}[a_1, \dots, a_m]$ denotes the divided difference of $e^x$ at the points $\{a_1, \dots, a_m\}$. Furthermore, using the linear transformation property for the fractional derivative, we have
\begin{equation}\notag
{_{-\infty}}D_{-\beta}^{\alpha-1}f(-\beta t)=t^{\alpha-1}{_{-\infty}}D_x^{\alpha-1}f(x), \quad x=-\beta t.
\end{equation}
Hence, we have the equivalent representation for the density of $Y$:
\begin{equation}
f_Y(t)=\frac{\beta_1^{\alpha_1} \dots \beta_k^{\alpha_k}t^{\alpha_1+\dots +\alpha_k-1}}{\Gamma(\alpha_1) \dots \Gamma(\alpha_k)}\{{_{-\infty}}D_{x_1}^{\alpha_1-1} \dots {_{-\infty}}D_{x_k}^{\alpha_k-1}\text{exp}[x_1,\dots , x_k]\}_{x_i=-\beta_i t, i=1,\dots, k}.
\end{equation}
\end{rem}

\begin{rem}
The Caputo definition for the fractional derivative, $_a{\widehat{D}}_x^{\nu}$, is perhaps the second most widely encountered definition, particularly in studies of fractional differential equations. The Caputo derivative of order $\nu>0$ is defined by:
\begin{equation}\notag
\begin{split}
_a{\widehat{D}}_x^{\nu}f(x)&=\frac{1}{\Gamma(m-\nu)}\int_a^x(x-\tau)^{m-\nu-1}f^{(m)}(\tau)d\tau, \quad \text{ for } m-1<\nu<m \\ 
&= {_a}I_x^{m-\nu}{_a}D_x^{m} f(x)={_a}I_x^{m-\nu}f^{(m)}(x) \\
&= {_a}D_x^{\nu}f(x)-\sum_{k=0}^{m-1}\frac{f^{(k)}(a)(x-a)^{k-\nu}}{\Gamma(k-\nu+1)}.
\end{split}
\end{equation}
See \cite{Podlubny1999}. It can be shown that for $m$ an integer, $lim\big({_a}{\widehat{D}}_x^{\nu}f(x)\big)$ = $f^{(m)}(x)$ as $\nu \to m$ and the definition is completed by defining $_a{\widehat{D}}_x^{m}f(x)$ = $f^{(m)}(x)$. The second line makes clear that the Caputo fractional derivative is composed as firstly an integer-order derivative of order $m\ge1$ followed by a RL fractional integral of order $m-\nu$. The RL definition has the reverse of this process and, as indicated earlier, does not always yield the same result. From the definition above, $_a{\widehat{D}}_x^{\nu}f(x)$=${_a}D_x^{\nu}f(x)$ if and only if $f^{(k)}(a)=0$, for $k=0, \dots, m-1$. However, setting $a=-\infty$ in the Caputo definitions then in the instance of $f(x)$ and all its derivatives being zero as $x \to -\infty$, $_{-\infty}{\widehat{D}}_x^{\nu}f(x)$ = ${_{-\infty}}D_x^{\nu}f(x)$ (see \cite{Podlubny1999} Sec. 2.4.1). Consequently, Proposition 8.1 is also valid under a Caputo fractional derivative definition with lower terminal $a=-\infty$.
\end{rem}

As a simple application, consider the density, $f_Y(t)$, for the sum, $Y=Z+X$, of two independent random variables where $Z$ has the gamma density $G_{\alpha, \beta}(t)$ and $X$ is exponential with parameter $\lambda<\beta$. We will make use of the following lemma.
\begin{lem}
For $f(x)=(1-e^x)/x$ and $x < 0$, ${_{-\infty}}I_x^{\nu}f(x)$ and ${_{-\infty}}D_x^{\nu}f(x)$ are given by
\begin{equation}\notag
{_{-\infty}}I_x^{\nu}f(x) = (-1)^{\nu}x^{\nu-1}\gamma(1-\nu,-x), \quad 0 < \nu < 1
\end{equation}
and
\begin{equation}\notag
{_{-\infty}}D_x^{\nu}f(x) = (-1)^{-\nu}x^{-(1+\nu)}\gamma(1+\nu,-x), \quad \nu > 0 .
\end{equation}
\end{lem}

\begin{proof}
Using the linearity property of ${_a}I_x^{\nu}$, we have 
\begin{equation}\notag
\begin{split}
{_{-\infty}}I_x^{\nu}f(x)&={_{-\infty}}I_x^{\nu}\Big(\frac{1}{x}\Big)-{_{-\infty}}I_x^{\nu}\Big(\frac{e^x}{x}\Big) \\
&=(-1)^\nu x^{\nu-1}\Gamma(1-\nu)-\frac{1}{\Gamma(\nu)}\int_{-\infty}^x(x-\tau)^{\nu-1}\frac{e^\tau}{\tau}d\tau \quad \text{ (from Lemma 8.1 part (i) and (8.3))} \\
&=(-1)^\nu x^{\nu-1}\Gamma(1-\nu)-\frac{(-1)}{\Gamma(\nu)}\int_0^{\infty}u^{\nu-1}\frac{e^{-(u-x)}}{u-x}du \quad \text{ ($\tau \mapsto x-u$)} \\
&=(-1)^\nu x^{\nu-1}\Gamma(1-\nu)-(-1)^\nu x^{\nu-1}\Gamma(1-\nu,-x),
\end{split}
\end{equation}
from the alternative integral representation for $\Gamma(a,z)$ in (3.11). The first part of the lemma then follows after noting that $\Gamma(a)-\Gamma(a,z)=\gamma(a,z)$.

For the second part
\begin{equation}\notag
\begin{split}
{_{-\infty}}D_x^{\nu}f(x)&=\frac{d^m}{dx^m}\Big({_{-\infty}}I_x^{m-\nu}f(x)\Big), \quad \text{ for } m-1< \nu < m \text{ and } m \in \mathbb{N}, \\
&=(-1)\frac{d^m}{dx^m}\frac{\gamma(1-m+\nu,-x)}{(-x)^{1-m+\nu}}= (-1)\frac{\gamma(1+\nu,-x)}{(-x)^{1+\nu}} \quad \text{ (using (3.8))} \\
&= (-1)^{-\nu}x^{-(1+\nu)}\gamma(1+\nu,-x)
\end{split}
\end{equation}
as required.
\end{proof}

Using the representation in (8.7), the density for $Y=Z+X$ for the instance of $\alpha>1$ is found as:
\begin{equation}\notag
\begin{split}
f_Y(t)&=\frac{\beta^\alpha \lambda t^{\alpha}}{\Gamma(\alpha)}\{{_{-\infty}}D_{x_1}^{\alpha-1}\text{exp}[x_1, x_2]\}_{x_1=-\beta t, x_2=-\lambda t} \\
&=\frac{\beta^\alpha \lambda t^{\alpha}}{\Gamma(\alpha)} \{ e^{x_2}(-1){_{-\infty}}D_{x_1}^{\alpha-1}\Big(\frac{1-e^{x_1-x_2}}{x_1-x_2}\Big) \} \\
&=\frac{\beta^\alpha \lambda t^{\alpha}}{\Gamma(\alpha)} \{ e^{x_2}(-1){_{-\infty}}D_y^{\alpha-1}\Big(\frac{1-e^y}{y}\Big) \}, \quad \text{ for } y=x_1-x_2 \\
&=\frac{\beta^\alpha \lambda t^{\alpha}}{\Gamma(\alpha)} \{ e^{x_2}(-1)^{-\alpha}y^{-\alpha}\gamma(\alpha,-y) \}_{x_1=-\beta t, x_2=-\lambda t} \quad \text{ (using Lemma 8.2) }\\
&=\frac{\beta^\alpha \lambda e^{-\lambda t}\gamma(\alpha,(\beta-\lambda)t)}{\Gamma(\alpha)(\beta-\lambda)^{\alpha}},\quad t>0,
\end{split}
\end{equation}
giving a generalisation of (6.4) to noninteger shape parameters and agreeing with Eqn.(6) of \cite{Nadarajah2005}. It is easily verified that the same result is found when $0 < \alpha <1$, using ${_{-\infty}}I_x^{1-\alpha}$ in place of ${_{-\infty}}D_x^{\alpha-1}$.

\subsection{The density for sums of independent gamma random variables}
For $Y=Z_1+Z_2$ where $Z_i$ has gamma density $G_{\alpha_i, \beta_i}(t)$ ($i=1,2$), we assume $\beta_2>\beta_1$ without loss of generality. We may find the density for $Y$ as:
\begin{equation}\notag
\begin{split}
f_Y(t)&=\frac{\beta_1^{\alpha_1}\beta_2^{\alpha_2}t^{\alpha_1+\alpha_2-1}}{\Gamma(\alpha_1)\Gamma(\alpha_2)}\{{_{-\infty}}D_{x_1}^{\alpha_1-1}{_{-\infty}}D_{x_2}^{\alpha_2-1}\text{exp}[x_1,x_2] \}_{x_1=-\beta_1t,x_2-\beta_2t}\\
&=\frac{\beta_1^{\alpha_1}\beta_2^{\alpha_2}t^{\alpha_1+\alpha_2-1}}{\Gamma(\alpha_1)\Gamma(\alpha_2)}\big\{{_{-\infty}}D_{x_1}^{\alpha_1-1}\Big(e^{x_1}\frac{\gamma(\alpha_2,x_1-x_2)}{(x_1-x_2)^{\alpha_2}}\Big)\big\}_{x_1=-\beta_1t,x_2-\beta_2t}
\end{split}
\end{equation}
and then continue by applying the fractional calculus version of Leibniz rule (see \cite{Podlubny1999} Sec. 2.7.2). Alternatively, we may take another route. 

The integral form for the divided difference $f[a_1,a_2]$ over the function $f$ may also be expressed as
\begin{equation}\notag
f[a_1,a_2]=\int_0^1f'(va_1+(1-v)a_2)dv, \quad a_i \in \mathbb{R}
\end{equation}
so that, for $x_1 > x_2$,
\begin{equation}\notag
\begin{split}
{_{-\infty}}D_{x_1}^{\alpha_1-1}{_{-\infty}}D_{x_2}^{\alpha_2-1}\text{exp}[x_1,x_2]&= {_{-\infty}}D_{x_1}^{\alpha_1-1}\{{_{-\infty}}D_{x_2}^{\alpha_2-1}\int_0^1e^{vx_1+(1-v)x_2}dv\}\\
&= \int_0^1{_{-\infty}}D_{x_1}^{\alpha_1-1}e^{vx_1}{_{-\infty}}D_{x_2}^{\alpha_2-1}e^{(1-v)x_2}dv\\
&= \int_0^1v^{\alpha_1-1}e^{vx_1}(1-v)^{\alpha_2-1}e^{(1-v)x_2}dv \\
&= e^{x_2}\int_0^1v^{\alpha_1-1}(1-v)^{\alpha_2-1}e^{(x_1-x_2)v}dv.
\end{split}
\end{equation}
Using (3.10), the integral representation for the confluent hypergeometric function, we may write
\begin{equation}\notag
\int_0^1v^{\alpha_1-1}(1-v)^{\alpha_2-1}e^{(x_1-x_2)v}dv = \frac{\Gamma(\alpha_1)\Gamma(\alpha_2)}{\Gamma(\alpha_1+\alpha_2)}M(\alpha_1,\alpha_1+\alpha_2,(x_1-x_2))
\end{equation}
and we may conclude that
\begin{equation}\notag
\begin{split}
f_Y(t)&=\frac{\beta_1^{\alpha_1}\beta_2^{\alpha_2}t^{\alpha_1+\alpha_2-1}}{\Gamma(\alpha_1)\Gamma(\alpha_2)}\{{_{-\infty}}D_{x_1}^{\alpha_1-1}{_{-\infty}}D_{x_2}^{\alpha_2-1}\text{exp}[x_1,x_2]\}|_{x_i=-\beta_i t, i=1,2} \\
&=\frac{\beta_1^{\alpha_1}\beta_2^{\alpha_2}t^{\alpha_1+\alpha_2-1}e^{x_2}}{\Gamma(\alpha_1+\alpha_2)}M(\alpha_1,\alpha_1+\alpha_2,(x_1-x_2))|_{x_i=-\beta_i t, i=1,2} \\
&=\frac{\beta_1^{\alpha_1}\beta_2^{\alpha_2}t^{\alpha_1+\alpha_2-1}e^{-\beta_2t}}{\Gamma(\alpha_1+\alpha_2)}M(\alpha_1,\alpha_1+\alpha_2,(\beta_2-\beta_1)t), \quad t>0
\end{split}
\end{equation}
giving a generalisation of (6.8) to noninteger shape parameters.

This approach lends itself to an easy extension for finding the density for the sum of $k$ independent gamma random variables, $k \ge 2$. We provide the \textit{Hermite-Genocchi theorem} for the integral form of the divided difference $f[a_1, \dots, a_n]$:
\begin{thm}
(\textit{Hermite-Genocchi}). Let $f \in C^{n-1}(\mathbb{R})$ and let $a_1, \dots, a_n$ be (not necessarily distinct) real numbers. Then, for $n \ge 2$,
\begin{equation}\notag
\begin{split}
f[a_1, \dots, a_n]&=\int_{\mathbb{S}_{n-1}}f^{(n-1)}(v_1a_1+ \dots + v_na_n)dv_1 \dots dv_{n-1} \\
&=\int_0^1dv_1\int_0^{1-v_1}dv_2 \dots \int_0^{1-\sum_{k=1}^{n-2}v_k}dv_{n-1}f^{(n-1)}(v_1a_1+ \dots v_na_n)
\end{split}
\end{equation}
where the domain of integration is the simplex
\begin{equation}\notag
\mathbb{S}_{n-1}=\big\{(v_1, v_2, \dots, v_{n-1}) \in \mathbb{R}_+^{n-1} : \sum_{i=1}^{n-1} v_i \le 1 \big\}
\end{equation}
and
\begin{equation}\notag
v_n=1-\sum_{i=1}^{n-1} v_i.
\end{equation}
\end{thm}
\begin{proof}
See, for example, \cite{Baxter2011} (noting that as $f[a_1,\ldots, a_n]$ is a symmetric function of its arguments, $f[a_1,\ldots, a_n] \equiv f[a_n, a_1, \ldots, a_{n-1}]$).
\end{proof}

We proceed by assuming, without loss of generality, that $\beta_k = \text{max}\{\beta_i, i=1, \dots, k\}$ so that $(x_i-x_k)>0$ for $i=1, \dots, k-1$. Next, apply Theorem 8.1 to $\text{exp}[x_1, \dots, x_k]$ to give
\begin{equation}\notag
\begin{split}
\text{exp}[x_1, \dots, x_k]&=\int_{\mathbb{S}_{k-1}}\text{exp}^{(k-1)}(v_1x_1+ \dots + v_kx_k)dv_1 \dots dv_{k-1} \\
&=\int_{\mathbb{S}_{k-1}}e^{v_1x_1+ \dots + v_kx_k}dv_1 \dots dv_{k-1} \\
&=\int_{\mathbb{S}_{k-1}}\Big(\prod_{i=1}^{k-1}e^{v_ix_i}\Big)e^{(1-\sum_{j=1}^{k-1}v_j)x_k}dv_1 \dots dv_{k-1}.
\end{split}
\end{equation}
The density for the sum of $k$ independent gamma random variables is then found as follows. Firstly,
\begin{equation}\notag
\begin{split}
{_{-\infty}}D_{x_1}^{\alpha_1-1}& \dots {_{-\infty}}D_{x_k}^{\alpha_k-1}\text{exp}[x_1, \dots, x_k] \\
&=\int_{\mathbb{S}_{k-1}}\Big(\prod_{i=1}^{k-1}{_{-\infty}}D_{x_i}^{\alpha_i-1}e^{v_ix_i}\Big){_{-\infty}}D_{x_k}^{\alpha_k-1}e^{(1-\sum_{j=1}^{k-1}v_j)x_k}dv_1 \dots dv_{k-1}\\
&=\int_{\mathbb{S}_{k-1}}\Big(\prod_{i=1}^{k-1}v_i^{\alpha_i-1}e^{v_ix_i}\Big)(1-\sum_{j=1}^{k-1}v_j)^{\alpha_k-1}e^{(1-\sum_{j=1}^{k-1}v_j)x_k}dv_1 \dots dv_{k-1} \\
&=e^{x_k}\int_{\mathbb{S}_{k-1}}\prod_{i=1}^{k-1}v_i^{\alpha_i-1}(1-\sum_{j=1}^{k-1}v_j)^{\alpha_k-1}e^{\sum_{j=1}^{k-1}v_j(x_j-x_k)}dv_1 \dots dv_{k-1} \\
&=\frac{\prod_{i=1}^k \Gamma(\alpha_i)}{\Gamma(\sum_{i=1}^k\alpha_i)}e^{x_k}\Phi_2^{(k-1)}\big(\alpha_1, \dots, \alpha_{k-1};\sum_{i=1}^k\alpha_i;(x_1-x_k), \dots, (x_{k-1}-x_k)\big),
\end{split}
\end{equation}
where, $\Phi_2^{(n)}$ denotes Erd\'elyi's confluent form of the fourth Lauricella function $F_D^{(n)}$ (see \cite{Srivastava1985} Sec.1.4) and where the multiple integral term in the penultimate line above, multiplied by $\Gamma(\alpha_1+ \dots +\alpha_k)/\prod_{i=1}^k \Gamma(\alpha_i)$, is recognised as being a representation for this confluent form. (See \cite{DiSalvo2008}).

Finally, we have the density for $Y=Z_1 + \dots +Z_k$
\begin{equation}\notag
\begin{split}
f_Y(t)&=\frac{\beta_1^{\alpha_1} \dots \beta_k^{\alpha_k}t^{\alpha_1+\dots +\alpha_k-1}}{\Gamma(\alpha_1) \dots \Gamma(\alpha_k)}{_{-\infty}}D_{x_1}^{\alpha_1-1}\dots {_{-\infty}}D_{x_k}^{\alpha_k-1}\text{exp}[x_1,\dots, x_k]|_{x_i=-\beta_i t, i=1, \dots, k} \\
&=\frac{\beta_1^{\alpha_1} \dots \beta_k^{\alpha_k}t^{\alpha_1+\dots +\alpha_k-1}}{\Gamma(\alpha_1+ \dots +\alpha_k)}e^{-\beta_kt}\Phi_2^{(k-1)}\big(\alpha_1, \dots, \alpha_{k-1};\sum_{i=1}^k\alpha_i;(\beta_k-\beta_1)t, \dots, (\beta_k-\beta_{k-1})t\big), \quad t>0
\end{split}
\end{equation}
which can be seen to be an equivalent expression to Eqn.(9) given in \cite{Mathai1982} and to Eqn.(6) in \cite{DiSalvo2008} and reduces to the expression for the density given earlier for $k=2$.

\appendix
\section{}
\subsection{Proof of Proposition 6.1} 
\begin{proof}
We prove this in two parts:

(a) Examine the differential of $f[a_1,a_2, \dots, a_k]$ w.r.t. $a_1$ as a limit as follows:
\begin{equation}\notag
\begin{split}
\frac{\partial}{\partial a_1}f[a_1,a_2, \dots, a_k]&= \lim_{h \to 0}\left(\frac{f[a_1+h, a_2, \dots, a_k]-f[a_1, a_2, \dots, a_k]}{h}\right)\\
&= \lim_{h \to 0}\left(\frac{f[a_1+h, a_2, \dots, a_k]-f[a_1, a_2, \dots, a_k]}{(a_1+h)-a_1}\right)\\
&=\lim_{h \to 0} f[a_1, a_2, \dots, a_k,a_1+h] \\
&=f[a_1^{(2)}, a_2, \dots, a_k].
\end{split}
\end{equation}

(b) Consider the derivative of $f[u_1^1, \dots,u_n^1,a_2, \dots, a_k]$, $n>1$, w.r.t. $a_1$ and where each argument $u_i^1$ is a function of $a_1$:
\begin{equation}\notag
\begin{split}
\frac{\partial}{\partial a_1}f[u_1^1, \dots,u_n^1,a_2, \dots, a_k]&= \sum_{i=1}^n \frac{\partial}{\partial u_i^1}f[u_1^1, \dots,u_n^1,a_2, \dots, a_k]\frac{du_i^1}{da_1}\\
&= \sum_{i=1}^n f[u_i^1, u_1^1, \dots,u_n^1, a_2, \dots, a_k]\frac{du_i^1}{da_1},
\end{split}
\end{equation}
using (a). Now define $u_i^1=a_1$ for $i=1, \dots, n$. Hence we have
\begin{equation}\notag
\frac{\partial}{\partial a_1}f[a_1^{(n)},a_2, \dots, a_k]= nf[a_1^{(n+1)}, a_2, \dots, a_k].
\end{equation}
Using (a) and (b), we therefore have
\begin{equation}\notag
\begin{split}
\frac{\partial^2}{\partial a_1^2}f[a_1, a_2, \dots, a_k]&=\frac{\partial}{\partial a_1}f[a_1^{(2)},a_2, \dots, a_k]\\
&= 2f[a_1^{(3)}, a_2, \dots, a_k].
\end{split}
\end{equation}
It follows then that (b) taken with (a), the proposition is proved. See also \cite{Hildebrand1956} Ch. 2.
\end{proof}

\subsection{Proof of Lemma 6.1}
\begin{proof}
We recall the general Leibniz rule for $m$ differentiable functions $g_i(x)$, $i=1, \dots, m$:
\begin{equation}\notag
\frac{\partial^n}{\partial x^n}\{g_1(x) \dots g_m(x)\} =\sum\limits_{\substack{r_1+\dots+r_m=n\\r_j \ge 0}}\frac{n!}{r_1! \dots r_m!}g_1^{(r_1)} \dots g_m^{(r_m)}
\end{equation}
and note that
\begin{equation}\notag
\frac{d^{k}}{d a^{k}}\Big(\frac{1}{a^n}\Big) =\frac{(-1)^k(n+k-1)!}{(n-1)!a^{n+k}}, \quad \text{ for } n \in \mathbb{N}.
\end{equation}
The partial derivative for the divided difference $f[a_1, \dots, a_m]$ in the lemma is then found by first invoking its Lagrange polynomial representation as follows:
\begin{equation}\notag
\begin{split}
\frac{\partial^k}{\partial a_1^{k_1} \dots \partial a_m^{k_m}}f[a_1, \dots , a_m] &= \frac{\partial^k}{\partial a_1^{k_1} \dots \partial a_m^{k_m}}\Big\{\sum_{i=1}^m f(a_i)\prod\limits_{\substack{q=1 \\ q\neq i}}^m\frac{1}{(a_i-a_q)}\Big\} \\
&=\sum_{i=1}^m \frac{\partial^{k_i}}{\partial a_i^{k_i}}\Big\{f(a_i)\prod\limits_{\substack{q=1 \\ q\neq i}}^m\frac{k_q!}{(a_i-a_q)^{k_q+1}}\Big\} \\
&= \sum_{i=1}^m k_i! \sum\limits_{\substack{r_1+\dots+r_m=k_i\\r_j \ge 0}} \frac{f^{r_i}(a_i)}{r_i!}\prod\limits_{\substack{q=1 \\ q\neq i}}^m\frac{(-1)^{r_q}(k_q+r_q)!}{(a_i-a_q)^{k_q+r_q+1}r_q!} \\
& \quad \text{ (using the general Leibniz rule)} \\
&= \sum_{i=1}^m k_i! \sum\limits_{\substack{r_1+\dots+r_m=k_i\\r_j \ge 0}} \frac{(-1)^{k_i-r_i}f^{r_i}(a_i)}{r_i!}\prod\limits_{\substack{q=1 \\ q\neq i}}^m\frac{(k_q+r_q)!}{(a_i-a_q)^{k_q+r_q+1}r_q!}.
\end{split}
\end{equation}
\end{proof}

\section*{Acknowledgements}
The author would like to thank Corina Constantinescu and Wei Zhu for helpful discussions. In addition, the author is pleased to acknowledge the valuable comments and suggestions of the two anonymous reviewers.
\\
\newline
This is a post-peer-review, pre-copyedit version of an article published in \textit{Statistical Papers}. The final authenticated version is available online at: https://doi.org/10.1007/s00362-021-01256-x

\section*{Declarations}
No financial support was provided for the conduct of the research and/or preparation of this manuscript.
\\
\\
The author has no conflicts of interest to declare that are relevant to the content of this manuscript.

\end{document}